\documentclass[11pt]{amsart} % use larger type; default would be 10pt

\usepackage[utf8]{inputenc} % set input encoding (not needed with XeLaTeX)

\usepackage{geometry} % to change the page dimensions
\usepackage{graphicx} % support the \includegraphics command and options

\usepackage{booktabs} % for much better looking tables
\usepackage{array} % for better arrays (eg matrices) in maths
\usepackage{paralist} % very flexible & customisable lists (eg. enumerate/itemize, etc.)
\usepackage{verbatim} % adds environment for commenting out blocks of text & for better verbatim
\usepackage{subfig} % make it possible to include more than one captioned figure/table in a single float

\usepackage{amsthm}
\usepackage{amssymb}
\usepackage{latexsym}
\usepackage{array}
\usepackage{dcolumn}% Align table columns on decimal point
\usepackage{amsfonts}% popular packages from the American Mathematical Society
\usepackage{bm}% Bold Math package

\theoremstyle{plain}
\newtheorem{theorem}{Theorem}
\newtheorem{proposition}[theorem]{Proposition}

\usepackage{amssymb,amsmath}
\usepackage{caption}
\usepackage{fancyhdr} % This should be set AFTER setting up the page geometry
\begin{document}

\title{Construction of an ordinary Dirichlet series with convergence beyond the Bohr strip }
\author{Brian N. Maurizi \\ 20 Walters Avenue \\ Staten Islnad, NY 10301 \\ brian.maurizi@gmail.com}

\date{} % Activate to display a given date or no date (if empty),
         % otherwise the current date is printed 

\maketitle

\begin{abstract}

An ordinary Dirichlet series has three abscissae of interest, describing the maximal regions where the Dirichlet series converges, converges uniformly, and converges absolutely.  The paper of Hille and Bohnenblust in 1931, regarding the region on which a Dirichlet series can converge uniformly but not absolutely, has prompted much investigation into this region, the ``Bohr strip."  However, a related natural question has apparently gone unanswered: For a Dirichlet series with non-trivial Bohr strip, how far beyond the Bohr strip might the series converge?  We investigate this question by explicit construction, creating Dirichlet series which converge beyond their Bohr strip.

\end{abstract}

\section{Introduction}

An ordinary Dirichlet series is a function of the form
$$
f(s) = \sum_{n=1}^{\infty} a_n n^{-s}
$$
with $s = \sigma + i t \in \mathbb{C}$.  The region on which a Dirichlet series might be expected to converge is a right half plane, we denote these by
$$
\Omega_{\sigma} = \{ s \in \mathbb{C} : \Re{s} > \sigma \}
$$
(where $\Re$ denotes the real part) and its closure will be written $\overline{\Omega}_{\sigma}$.  To a Dirichlet series we can associate several abscissae:
\begin{align}
\sigma_a &= \inf \{ \sigma : \sum a_n n^{-s} \text{ converges absolutely for } s \in \Omega_\sigma \}  \nonumber \\
\sigma_b &= \inf \{ \sigma : \sum a_n n^{-s} \text{ converges to a bounded function on } \Omega_\sigma \}  \nonumber \\
\sigma_c &= \inf \{ \sigma : \sum a_n n^{-s} \text{ converges for all } s \in \Omega_\sigma \}  \; . \nonumber
\end{align}
From the definitions, it is evident that $\sigma_c \le \sigma_b \le \sigma_a$.  Harald Bohr proved that $\sigma_a - \sigma_b \le 1/2$ in \cite{Bohr_1913_2}.  In 1931, Hille and Bohnenblust \cite{Hille_Bohnenblust} showed that this is sharp; there exist Dirichlet series for which $\sigma_a - \sigma_b = 1/2$, and an explicit construction is provided in \cite{Hille_Bohnenblust}.

Let $\Lambda(n)$ be the number of prime factors of $n \in \mathbb{N}$, counted with multiplicity (so $\Lambda(8) = 3$).  In \cite{Hille_Bohnenblust} they also show that if $\sum a_n n^{-s}$ contains only terms of homogeneity at most $M$, i.e.
$$
\Lambda(n) > M \; \implies \; a_n = 0
$$
then we have $\sigma_a - \sigma_b \le \frac{1}{2} - \frac{1}{2M}$, and this is also sharp, which is shown by construction.

Since the publication of \cite{Hille_Bohnenblust}, there has been much investigation into the gap $\sigma_a - \sigma_b$, and the associated ``Bohr strip" $\{ s \in \mathbb{C} : \sigma_b < \Re{s} < \sigma_a \}$ and related issues, we recall some of them here.  We would like to mention a survey article in this area by Defant and Schwarting \cite{Defant_Survey}.

A key inequality in \cite{Hille_Bohnenblust} is the following inequality for $M$-homogenous polynomials:  For each $M$, there is a constant $D_M$ such that, for an $M$-homogenous polynomial $\sum_{|\alpha| = M} a_{\alpha} z^{\alpha}$ on $\mathbb{C}^n$ we have
\begin{equation}\label{E:M_homogenous_poly_estimate}
\left(  \sum_{|\alpha| = M} | a_{\alpha}|^{ \frac{2M}{M+1} } \right)^{ \frac{M+1}{2M}  } \le D_M \sup_{z \in \mathbb{D}^n } \left| \sum_{|\alpha| = M} a_{\alpha} z^{\alpha} \right|
\end{equation}
where $\alpha = (\alpha_1 , \ldots , \alpha_n)$ is a multi-index, $|\alpha| = \alpha_1 + \cdots + \alpha_n$, and $\mathbb{D}$ is the unit disc.  This is a generalization of the Littlewood 4/3 inequality \cite{Littlewood}, which proves the above result in the case $M=2$.  The original proof in \cite{Hille_Bohnenblust} did give a bound on the best possible $D_M$, but this bound has been substantially improved, see the discussions in \cite{Defant_et_al_Hypercontractive} and \cite{Diniz_et_al} which also contain the most recent improvements to our knowledge.  

Another development related to the question of the gap $\sigma_a - \sigma_b$ is the theory of $p$-Sidon sets (see the discussions in \cite{Queffelec_OldAndNew}, \cite{Defant_et_al_Hypercontractive}), the inequality (\ref{E:M_homogenous_poly_estimate}) shows that the set of monomials $\{ z^{\alpha} : |\alpha| = M \}$ is a $\frac{2M}{M+1}$-Sidon set, for example.

To produce Dirichlet series with a large gap $\sigma_a - \sigma_b$, in addition to the explicit construction of \cite{Hille_Bohnenblust}, random methods have been employed, such as in \cite{hartman} where the existence of ordinary Dirichlet series with $\sigma_a - \sigma_b = 1/2$, (and $\sigma_a - \sigma_b = \frac{1}{2} - \frac{1}{2M}$ for homogeneity $M$), is shown.  Random methods are also employed in sections 4 and 5 of \cite{Queffelec_OldAndNew}, to construct Dirichlet polynomials with small $\| \cdot \|_{\infty}$ norm and thus obtain bounds on the $1$-Sidon constant of the set of ``frequencies" $\{ \log 1, \ldots , \log N\}$ (interpreted as functions on the Bohr compactification of $\mathbb{R}$).

We mention these developments to note that, for all of this progress, a rather natural question remains: For a Dirichlet series with the ``gap'' $\sigma_a - \sigma_b$ being large, what can be said about the ``gap'' $\sigma_b - \sigma_c$ for this same series?  To what extent can a Dirichlet series converge beyond its Bohr strip?  This is the question we explore here.

We first present a general construction of an ordinary Dirichlet series of homogeneity $M$.  Our technique is based on the method of Walsh matrices used in \cite{Maurizi_Queffelec}, although we depart from \cite{Maurizi_Queffelec} by using non-square matrices.  We then prove four bounds on the abscissae of this series, of the following forms:
\begin{itemize}
\item{ $\sigma_b \le B$}
\item{ $ \sigma_b \ge 0$}
\item{ $\sigma_a \ge A$}
\item{ $\sigma_c \le C$} \; .
\end{itemize}
The construction only yields a non-trivial result (i.e. $\sigma_a - \sigma_b > 0 \; , \; \;\; \sigma_b - \sigma_c > 0 $)  for the cases $M=2,3$.  We present the construction for general $M$ nevertheless, because the exposition would not be much clearer for $M=3$ rather than general $M$, and because we hope that better bounds might be proved for the general construction which would then yield results beyond $M=3$.  For the cases $M=2,3$, we obtain:

\subsection*{$M=2$}

We construct a Dirichlet series of homogeneity $M=2$ satisfying
\begin{equation}\label{E:Mequals2_gap}
\sigma_a - \sigma_b = 1/4 \; , \;\;\; \sigma_b - \sigma_c \ge 1/4 \; .
\end{equation}
Construction of such a Dirichlet series (or even proof of its existence) is, to our knowledge, a new result.  Note that $1/4$ is the optimal value of $\sigma_a - \sigma_b$, given that $M=2$.

\subsection*{$M=3$}

Here, for any value $\rho_1 \in (0,1)$, we construct a Dirichlet series of homogeneity $M=3$ which satisfies $\sigma_a - \sigma_c \ge 1/3$, and we furthermore have some specific control over $\sigma_b \in (\sigma_c , \sigma_a)$:
$$
\sigma_a - \sigma_b \ge \frac{1+\rho_1}{6}   \; , \;\;\; \sigma_b - \sigma_c \ge \frac{ 1- \rho_1 }{9}\; .
$$
For $\rho_1 > 1/2$, we see that the value of $\sigma_a - \sigma_b$ is larger than $1/4$, so this construction does represent a result that cannot be achieved with only terms of homogeneity at most two.  If we pick, for example, $\rho_1 = 3/4$, then we have
$$
\sigma_a - \sigma_b \ge 7/24 \; , \;\;\; \sigma_b - \sigma_c \ge 1/36 \; .
$$

Note that, for $M=3$, unfortunately the current construction does not produce values for $\sigma_a - \sigma_b$, $\sigma_b - \sigma_c$ that couldn't be replicated by a Dirichlet series with existing constructions.  For instance, using the standard Hille-Bohnenblust construction for $M=3$ and adding a properly shifted Dirichlet series with a given value of $\sigma_a - \sigma_c$ (such as the alternating zeta function) will produce a series that has these properties.  This can be done only using terms of homogeneity three as well; simply use a version of the alternating zeta function which contains only terms of homogeneity three (and is ``alternating'' on these terms), we leave details to the interested reader.  However, such a series, being more simply constructed, does not afford control on the individual coefficients.  To our knowledge, ours is the first construction which exhibits a Dirichlet series having terms of homogeneity exactly three for which it is proved that $\sigma_a - \sigma_b > 1/4$, $\sigma_b - \sigma_c > 0$ and for which we have substantial knowledge regarding the individual coefficients.  

Our hope is that the method shown here, since it gives specific control over each abscissa, without any ``tricks" of adding another (unrelated) Dirichlet series, could be extended, specifically by improving the estimate in section \ref{S:General_InnerSum}.  

In section \ref{S:GenConstruct} we present the general construction.  In sections \ref{S:sigma_b} and \ref{S:sigma_a}, we prove the ``easy" bounds: an upper bound on $\sigma_b$ and a lower bound on $\sigma_a$.  These first two bounds yield the classic Hille-Bohnenblust-type Dirichlet series, for each homogeneity $M$.  In sections \ref{S:sigma_b_lower}, \ref{S:General_InnerSum} and \ref{S:General_sigma_c} we prove the ``hard" bounds, showing that our Dirichlet series is unbounded on any $\Omega_{\sigma}$ with $\sigma<0$, and then showing that our Dirichlet series converges (i.e. the sequence of partial sums converges) at a point $s=-\epsilon$ on the negative real axis.  Once all four bounds are proved, in section \ref{S:results} we derive the results for $M=2, 3$.

In section \ref{S:General_InnerSum} we isolate one of the key estimates, a basic size estimate on all partial sums of a certain set of complex numbers of modulus one.  It seems some improvement should be possible, given that the arguments are spread over the unit circle.

\section{General Construction}\label{S:GenConstruct}

For a polynomial $Q$ in complex variables $z_1 , z_2 , \ldots $, define the Wiener norm $\| Q \|_W$ to be the sum of the absolute values of the coefficients of $Q$, and let
$$
\| Q \|_{\infty} = \max \{ |Q| : |z_i| \le 1 \} \; .
$$
Let $\| \cdot \|$ denote the usual norm $\| x \|^2 = \sum |x_i|^2$.

The construction here differs from standard constructions of this type, because the matrices we use will not necessarily be square.

For $r \in \mathbb{N}$, let $\omega = \omega_r$ be the primitive $r$th root of unity $e^{2 \pi i / r}$.  For $r_1 \le r_2$, let $B^{(r_2, r_1)} : \mathbb{C}^{r_1} \rightarrow \mathbb{C}^{r_2}$ be the ``Walsh matrix'' defined by
$$
b_{ij} = \omega_{r_2}^{i j} \; , \;\;\;\; \; i = 0, 1, \ldots , r_2-1 \; , \; j = 0, 1, \ldots , r_1-1 \;\; .
$$
So, for $v \in \mathbb{C}^{r_1}$, $B^{(r_2, r_1)}$ satisfies $\| B^{(r_2, r_1)} v \|^2 = r_2 \| v \|^2$.

Let $M \in \mathbb{N}$, and suppose $r_1 \le \ldots \le r_{M+1}$.  Suppose we have $M$ sets of complex numbers, with the $j$th set having $r_j$ elements:
$$
z_0^{(1)}, \ldots , z_{r_1-1}^{(1)} , z_0^{(2)} , \ldots , z_{r_2-1}^{(2)}, \ldots , z_0^{(M)} , \ldots , z_{r_M-1}^{(M)} \;\;.
$$
Let $D^{(j)}$ be the $r_j \times r_j$ diagonal matrix with the diagonal entry $d_{i i} = z_i^{(j)}$.  We will abbreviate
$$
B^{2,1} = B^{(r_2, r_1)} \; , \;\;\; B^{3,2} = B^{(r_3, r_2)} \;\;\; \text{etc.}
$$
Let $u=(1, \ldots , 1) \in \mathbb{C}^{r_1}$, and consider the vector
$$
B^{M+1,M} D^{(M)} B^{M,M-1} D^{(M-1)} \cdots B^{3,2} D^{(2)} B^{2,1} D^{(1)} u \;\; \in \mathbb{C}^{r_{M+1}} \; .
$$
Suppose that each $z_i^{(j)}$ satisfies $ |z_i^{(j)}| \le 1$.  Then we have
\begin{align*}
\| B^{M+1,M} D^{(M)} \cdots  B^{2,1} D^{(1)} u \|^2 &= r_{M+1} \| D^{(M)} \cdots  B^{2,1} D^{(1)} u \|^2  \\
&\le r_{M+1} \|B^{M,M-1} D^{(M-1)} \cdots B^{3,2} D^{(2)} B^{2,1} D^{(1)} u \|^2 \\
&= \ldots \\
&= r_{M+1} r_M \cdots r_2 \| u \|^2 \\
&= r_{M+1} r_M \cdots r_2 r_1 \\
&= \Pi_1^{M+1} r_j \;\; .
\end{align*}
This means that each coordinate of $B^{M+1,M} D^{(M)} \cdots  B^{2,1} D^{(1)} u$ has size less than or equal to $ \left( \Pi_1^{M+1} r_j \right)^{1/2}$.  In particular the $0$th coordinate, which is
$$
\sum_{i_1=0}^{r_1-1} \cdots \sum_{i_M=0}^{r_M-1} z_{i_1}^{(1)} z_{i_2}^{(2)} \cdots z_{i_M}^{(M)} \omega_{r_2}^{ i_1 i_2} \omega_{r_3}^{ i_2 i_3} \cdots  \omega_{r_M}^{ i_{M-1} i_M } 
$$
satisfies
$$
\left| \sum_{i_1}^{r_1-1} \cdots \sum_{i_M}^{r_M-1} z_{i_1}^{(1)} \cdots z_{i_M}^{(M)} \omega_{r_2}^{ i_1 i_2}  \cdots  \omega_{r_M}^{ i_{M-1} i_M }  \right|  \le \left( \Pi_1^{M+1} r_j \right)^{1/2} \qquad \text{when} \;\; |z_i^{(j)}| \le 1  \; .
$$
We define
\begin{equation}\label{E:Define_Q}
Q = Q_{r_1, \ldots, r_M} =\sum_{i_1}^{r_1-1} \cdots \sum_{i_M}^{r_M-1} z_{i_1}^{(1)} \cdots z_{i_M}^{(M)} \omega_{r_2}^{ i_1 i_2}  \cdots  \omega_{r_M}^{ i_{M-1} i_M }  \;\; .
\end{equation}
Considering $Q$ as a polynomial in the variables $z_i^{(j)}$, we have
\begin{align}
\| Q \|_W &= \Pi_1^M r_j \label{E:WienerNormOfQ} \\
\| Q \|_{\infty} &\le \left( \Pi_1^{M+1} r_j \right)^{1/2} \;\; . \label{E:InfinityNormOfQ}
\end{align}
Note that $Q$ is not just a polynomial in the $z_i^{(j)}$, but is in fact a multilinear function of the vectors $z^{(1)}, z^{(2)}, \ldots , z^{(M)}$.  Note that $Q$ depends on $r_1 , \ldots, r_M$, so in our case, where the goal is to make the Wiener norm of $Q$ as large as possible, and make the supremum norm as small as possible, we set $r_{M+1}$ to the smallest possible value, namely $r_M$.

We denote the floor function by $ \lfloor x \rfloor = \max \{ n \in \mathbb{N} : n \le x \}$.  Next, we define some key parameters of this construction:
\begin{align*}
L &\in \mathbb{N} \\
\rho_1 &\le \ldots \le \rho_{M-1} \in [0,1] \;\;\; , \; \rho_M = 1  \\
r_1 &= \lfloor 2^{\rho_1 L} \rfloor , \ldots ,  r_M = \lfloor 2^{\rho_M L}  \rfloor   
\end{align*}
where $\rho_1 , \ldots , \rho_M $ are fixed and $L$ is an index which will range over $\mathbb{N}$.  Notice that the $r_j$ depend on $L$, so to be proper we might write $r_j^{(L)}$; we will not do so since the value of $L$ will be clear from context.
We define
$$
Q^L = Q_{ r_1 , \ldots , r_{M} } \; .
$$

When one has a polynomial $Q$ in the complex variables $z_1, z_2, \ldots$, we can create a Dirichlet polynomial $P$ via the substitution
$$
P(s) =  Q( 2^{-s} , 3^{-s} , \ldots ) \; .
$$
We have just defined a family of polynomials $\{ Q^L \}$, each being homogenous of degree $M$.  To create the Dirichlet polynomials and Dirichlet series that we want, we will use polynomials $Q$ from this family.

Let $p_k$ be the $k$th prime number.  Note that, by the Prime Number Theorem \cite{Hardy_Wright} (in fact, what we need is weaker than the Prime Number Theorem), we have $c,C$ such that
$$
c k \log k \le p_k \le C k \log k \; .
$$
This is all we will need about the distribution of primes.

Now, we define a dyadic family of disjoint sets of primes: Let the sets $K_L^{(j)}$ be defined by
$$
K_L^{(j)} = \left\{ (M+j-1) 2^L+x \; : \; x = 0 , \ldots , r_j -1 \right\}
$$
and then the family of sets of primes is defined by
$$
\Pi_L^{(j)} = \{ p_k : k \in K_L^{(j)} \} \; .
$$
For convenience, when the value of $L$ is clear from context, we denote the $i$th element of $K_L^{(j)}$ by
$$
k_i^{(j)} = (M+j-1) 2^L +i \; .
$$
Note that all of the $\Pi_L^{(j)}$ are pairwise disjoint.

The terms in the Dirichlet polynomial $P_L$ will involve only those $n$ which are a product of a single element from each $\Pi_L^{(j)}$.  Define
$$
\Pi_L^{\times} = \Pi_L^{(1)} \cdot \Pi_L^{(2)} \cdots \Pi_L^{(M)} = \left\{ n =  p_{k_{i_1}^{(1)}}  p_{k_{i_2}^{(2)}}  \cdots p_{k_{i_M}^{(M)}} \; , \;\;\; i_j  \in \{0, \ldots, r_j-1\} , \;  L \in \mathbb{N} \right\}
$$
and define
$$
P_L(s) = Q^L \left( p_{k_0^{(1)} }^{-s} , \ldots ,  p_{k_{r_1-1}^{(1)} }^{-s}  ,  p_{k_0^{(2)} }^{-s} , \ldots ,  p_{k_{r_2-1}^{(2)} }^{-s}  , \ldots,  p_{k_0^{(M)} }^{-s} , \ldots ,  p_{k_{r_M-1}^{(M)} }^{-s}   \right) \; .
$$
In other words,
$$
P_L = \sum_{n \in \mathbb{N}} \gamma_n n^{-s}
$$
where
$$
\gamma_n = \begin{cases} \omega_{r_2}^{ i_1 i_2}  \cdots  \omega_{r_M}^{ i_{M-1} i_M }  \qquad \text{ if } n \in \Pi_L^{\times} ,  \;\;\; n = p_{k_{i_1}^{(1)}} \cdots p_{k_{i_M}^{(M)}} \\ 0 \text{ else } \end{cases} \; .
$$
At this point, the idea is to consider the Dirichlet series $\sum_L c_L P_L$ with some coefficients $c_L$.  However, instead of defining the series in this way, we will define it ``directly'' by defining its coefficients $a_n$.  This will be convenient since we want to consider the conditional convergence with proper care.

We will let $X>0$ be a fixed real number which is not yet specified, but $X$ will only depend on $M$ and $\rho_1 , \ldots , \rho_M$.  So, consider the Dirichlet series
\begin{equation}\label{E:MainSeries}
f(s) = \sum_{n = 1}^{\infty} a_n n^{-s}
\end{equation}
where
\begin{align}
\beta_L &  \qquad \text{ is fixed but to-be-determined, with } |\beta_L|=1 \nonumber \\
a_n &= \begin{cases} \beta_L 2^{-X L} \;  \omega_{r_2}^{ i_1 i_2}  \cdots  \omega_{r_M}^{ i_{M-1} i_M }  \qquad \text{ if } n \in \Pi_L^{\times} ,  \;\;\; n = p_{k_{i_1}^{(1)}} \cdots p_{k_{i_M}^{(M)}} \\ 0 \text{ else } \end{cases} \; \; .  \label{E:Define_a_n}
\end{align}
So, the idea is
$$
f(s) = \sum_{L=1}^{\infty} \beta_L 2^{-X L} P_L(s)
$$
although (\ref{E:MainSeries}), (\ref{E:Define_a_n}) is the proper definition.

Having this general construction, we will now prove four bounds on the abscissae of $f$, in the next five sections.  We begin with the two easier bounds: an upper bound on $\sigma_b$ and a lower bound on $\sigma_a$.

\section{An upper bound on the abscissa of boundedness}\label{S:sigma_b}

Let $\sigma > 0$, let the real part of $s$ be greater than or equal to $\sigma$, and let $n_* $ be the smallest $n$ in $\Pi_L^{\times}$,
$$
n_* = p_{k_0^{(1)}}  p_{k_0^{(2)}}  \cdots p_{k_0^{(M)}} \; .
$$
We have
\begin{align*}
P_L(s) &= \sum_n \gamma_n n^{-s}  \\
&= n_*^{-\sigma} \sum_n \gamma_n n^{-s} n_*^{\sigma} \\
&= n_*^{-\sigma}  Q^L \left( p_{k_0^{(1)}}^{\sigma} p_{k_0^{(1)} }^{-s} \;  , \; \ldots , \; p_{k_0^{(1)}}^{\sigma} p_{k_{r_1-1}^{(1)} }^{-s} \; , \; \ldots, p_{k_0^{(M)}}^{\sigma} p_{k_0^{(M)} }^{-s} \; , \; \ldots , p_{k_0^{(M)}}^{\sigma} p_{k_{r_M-1}^{(M)} }^{-s}  \right)
\end{align*}
and
$$
| p_{k_0^{(j)}}^{\sigma} p_{k_i^{(j)} }^{-s} | \le 1 \;\;\; \forall i, j \; .
$$
Recall (\ref{E:InfinityNormOfQ}):
$$
\| Q^L \|_{\infty} \le 2^{ ( \rho_1 + \cdots + \rho_M + 1 ) L/2} \; .
$$
We also have $n_* \ge c \left( M 2^L \right)^M \ge c_M 2^{LM}$ by the Prime Number Theorem, so
that
$$
| P_L(s) | \le c_M^{-\sigma} 2^{-\sigma LM} 2^{ ( \rho_1 + \cdots + \rho_M + 1 ) L/2}
$$
and therefore
\begin{align*}
| \beta_L 2^{-X L}  P_L(s) | &\le c_M^{-\sigma} 2^{-X L} 2^{-\sigma LM} 2^{ ( \rho_1 + \cdots + \rho_M + 1 ) L/2} \\
&= c_M^{-\sigma} 2^{\big[ \; (1/2) ( \rho_1 + \cdots + \rho_M + 1 ) - \sigma M - X   \; \big] L} \; .
\end{align*}
We see that, if $(1/2) ( \rho_1 + \cdots + \rho_M + 1 ) - \sigma M - X < 0$, then $\sum_L \beta_L 2^{-X L} P_L$ defines a bounded holomorphic function in the half plane $\Omega_{\sigma}$. 

By inspection in some possibly distant right half plane we can conclude that 
$$
f = \sum_L \beta_L 2^{-X L} P_L
$$
in this distant right half plane (since the two Dirichlet series will converge absolutely).  So, $\sum_L \beta_L 2^{-X L} P_L$ gives an analytic continuation of $f$ to a bounded function on $\Omega_{\sigma}$, and therefore by a classic theorem of Bohr \cite{Bohr_1913_1} we know that the Dirichlet series for $f$ converges on $\Omega_{\sigma}$, and $f$ is bounded there, so  $\sigma_b \le \sigma$.

We have shown that if $\sigma > 0$ and $(1/2) ( \rho_1 + \cdots + \rho_M + 1 ) - \sigma M - X < 0$ then $\sigma_b \le \sigma$.  So, if 
$$
X \le (1/2) ( \rho_1 + \cdots + \rho_M + 1 )
$$ 
and we choose any $\sigma$ satisfying 
$$
\sigma > (1/2M) ( \rho_1 + \cdots + \rho_M + 1 )  - X/M
$$
then $\sigma_b \le \sigma$, and therefore by taking the infimum over $\sigma$ we have proved:

\begin{proposition}\label{P:sigma_b_upperbound}
Let $f$ be the Dirichlet series defined by  (\ref{E:MainSeries}) and (\ref{E:Define_a_n}).  If
$$
X \le (1/2) ( \rho_1 + \cdots + \rho_M + 1 )
$$
then we have
$$
\sigma_b \le (1/2M) ( \rho_1 + \cdots + \rho_M + 1 )  - X/M \; .
$$
\end{proposition}

\section{Abscissa of absolute convergence}\label{S:sigma_a}

We prove a lower bound on $\sigma_a$.  Note that, by the Prime Number Theorem, we have $c,C$ such that
$$
c k \log k \le p_k \le C k \log k
$$
and so we have
$$
\max \{ n : n \in \Pi_L^{\times} \} \le  \left( \max \{ p_k : k \in K_L^{(j)} \}  \right)^M  \; .
$$
We recall that all $k \in K_L^{(j)}$ satisfy $k \le M 2^{L+1}$, and so
$$
p_k \le M 2^{L+2} ( \log M + (L+1) \log 2) \; .
$$
Therefore, with $C_M$ being a constant depending on $M$,
$$
\max \{ p_k : k \in K_L^{(j)} \} \le C_M 2^L L
$$
and so
\begin{equation}\label{E:bound_n_L}
\max \{ n : n \in \Pi_L^{\times} \} \le C_M^M 2^{M L} L^M \; .
\end{equation}
Also, we have
$$
\left| \Pi_L^{\times} \right| = r_1 \cdots r_M \ge \frac{ r_{M+1}^{\rho_1} }{2} \cdots \frac{ r_{M+1}^{\rho_M} }{2}  \ge 2^{-M} 2^{ (\rho_1 + \cdots + \rho_M) L} \; .
$$
Therefore, we calculate:
\begin{align*}
\sum |a_n| n^{-\sigma} &= \sum_{L=1}^{\infty} 2^{-X L}  \sum_{n \in \Pi_L^{\times} } n^{-\sigma} \\
&\ge \sum_{L=1}^{\infty} 2^{-X L} \left| \Pi_L^{\times} \right|  C_M^{-\sigma}  2^{-\sigma M L} L^{-\sigma M} \\
&\ge C_M^{-\sigma} 2^{-M} \sum_{L=1}^{\infty} 2^{-X L} 2^{ (\rho_1 + \cdots + \rho_M) L}  2^{-\sigma M L} L^{-\sigma M} \\
&= C_M^{-\sigma} 2^{-M} \sum_{L=1}^{\infty} 2^{ \big[ \; (\rho_1 + \cdots + \rho_M) - \sigma M - X  \; \big] L} L^{-\sigma M}  \; .
\end{align*}
If $(\rho_1 + \cdots + \rho_M) - \sigma M - X > 0$, i.e. if 
$$
\sigma < (1/M) (\rho_1 + \cdots + \rho_M) - X/M
$$
then the above sum is infinite, so we have

\begin{proposition}\label{P:sigma_a}
Let $f$ be the Dirichlet series defined by  (\ref{E:MainSeries}) and (\ref{E:Define_a_n}).  We have 
$$
\sigma_a \ge (1/M) (\rho_1 + \cdots + \rho_M) - X/M \; .
$$ 
\end{proposition}

At this point, we note that we have produced the classic Hille-Bohnenblust construction.
With propositions \ref{P:sigma_b_upperbound} and \ref{P:sigma_a}, we see that as long as we choose
$$
X \le (1/2) ( \rho_1 + \cdots + \rho_M + 1 )
$$
then we have
$$
\sigma_a - \sigma_b \ge \frac{1}{2M} ( \rho_1 + \cdots + \rho_M ) - \frac{1}{2M} \; .
$$
For any value of $M$, by choosing $\rho_1 = \ldots = \rho_M = 1$ (and $X=0$ for instance), the Dirichlet series $f$ has terms of homogeneity exactly $M$ and $\sigma_a - \sigma_b \ge \frac{1}{2} - \frac{1}{2M}$, the largest possible gap between $\sigma_a$ and $\sigma_b$.

\section{Proving $\sigma_b \ge 0$}\label{S:sigma_b_lower}

To show that $f$ becomes unbounded if we cross the abscissa $\sigma=0$, i.e. to prove $\sigma_b \ge 0$, we will demonstrate that (under certain conditions on $X$) the partial sums of $f$ achieve arbitrarily large values on the imaginary axis.  This proves the bound because, if $\sigma_b < 0$ then by classic results \cite{Bohr_1913_1} the partial sums of $f$ converge uniformly to $f$ on the line $\sigma = 0$.  Uniform convergence implies that there is some large $N'$, such that
$$
\forall N \ge N', \| \sum_{N=1}^{N'} a_n n^{-it} \|_{\infty} \le 2 \| f \|_{\infty} \; .
$$
In particular, with $\sigma_b < 0$ the partial sums cannot achieve arbitrarily large values.  Note that we will not show that $f$ itself is actually unbounded on the imaginary axis (this is fortunately not necessary; it may well not be true).

To prove that the partial sums achieve arbitrarily large values on the imaginary axis, we will show that the first finitely many $L$ of the terms $\beta_L 2^{-X L} P_L$ have very similar arguments at some point on the imaginary axis.  Furthermore, we want to prove this while still retaining complete flexibility to choose $\{\beta_L\}$.   This can be accomplished because, for any choice of $\{ \beta_L \}$, once we consider a particular partial sum, we can search the entire imaginary axis to choose the argument of $ P_L$ at will, for each of finitely many $L$ (simultaneously) using Kronecker's theorem.

The plan is as follows:

\begin{enumerate}

\item{ Fix a large number $K$; the partial sums will then be constructed to exceed $K$. }

\item{ Let $L_K = 2^{M+1} K$; this is chosen so $\sum_{L \le  L_K} 2^{-X L} \| Q^L \|_{\infty} > 2 K$ (for the right $X$). }

\item{ For each $L \le  L_K$, find $z_L$ such that $\beta_L Q^L( z_L ) = \| Q^L \|_{\infty} $.} 

\item{ By Kronecker's theorem, find $t = t_K$ such that $\left( p_{k_0^{(1)} }^{-i t_K} , \ldots , p_{k_{r_M-1}^{(M)} }^{- i t_K}  \right)  \approx z_L$ is satisfied simultaneously for all $L \le L_K$. }

\end{enumerate}

Then, $P_L ( i t_K ) \approx Q^L (z_L)$ and so 
\begin{align*}
\sum_{L \le L_K} 2^{-X L} \beta_L P_L( i t_K) &\approx \sum_{L \le L_K} 2^{-X L} \beta_L Q^L (z_L) \\
&= \sum_{L \le  L_K} 2^{-X L} \| Q^L \|_{\infty}  
\end{align*}
where the first sum is indeed a partial sum of the Dirichlet series for $f$ at the point $s = i t_K$.  The fact that this is in fact a partial sum of the Dirichlet series for $f$ is not a complete triviality, but it is true; it equals $\sum_{n=1}^{N} a_n n^{-i t_K}$ where $N = \max \{ n : n \in \Pi_{L_K}^{\times} \}$.  The important point is that, if $L_1 < L_2$ then
$$
\forall n_1 \in \Pi_{L_1}^{\times} \; , \; \forall n_2 \in \Pi_{L_2}^{\times}  \; , \; \text{we have} \;\; n_1 < n_2 \; .
$$

Now, the details.  We first fix a (large) K.

Recall that
$$
Q^L = \sum_{i_1}^{r_1-1} \cdots \sum_{i_M}^{r_M-1} z_{i_1}^{(1)} \cdots z_{i_M}^{(M)} \omega_{r_2}^{ i_1 i_2}  \cdots  \omega_{r_M}^{ i_{M-1} i_M } \; .
$$ 
Recall inequality (\ref{E:M_homogenous_poly_estimate}) from the introduction: For any $M$-homogenous polynomial in $n$ variables, we have
$$
\left(  \sum_{|\alpha| = M} | a_{\alpha}|^{ \frac{2M}{M+1} } \right)^{ \frac{M+1}{2M}  } \le D_M \sup_{z \in \mathbb{D}^n } \left| \sum_{|\alpha| = M} a_{\alpha} z^{\alpha} \right| \; .
$$
By the maximum modulus principle, $Q^L$ will attain its supremum on $\mathbb{T}^n$.  So, for each $L$, let $( z_i^{* (j)} ) = ( z_i^{*(j)}(L) )$ be a point on $\mathbb{T}^n$ where the supremum is attained.  We have
\begin{align}
\left| Q^L \left( ( z_i^{*(j)} ) \right) \right| \ge D_M^{-1} \left(  \sum_{|\alpha| = M} | a_{\alpha}|^{ \frac{2M}{M+1} } \right)^{ \frac{M+1}{2M}  } &= D_M^{-1} \left(  r_1 \cdots r_M \right)^{ \frac{M+1}{2M}  }  \nonumber \\
&\ge 2^{-M} 2^{ ( \rho_1 + \cdots + \rho_M ) \frac{M+1}{2M} L}  \; .  \label{E:z_star}
\end{align}
Next, note that since $Q^L$ is in fact multilinear on $\mathbb{C}^{r_1} \times \cdots \times \mathbb{C}^{r_M}$, we can multiply by a constant in the $\mathbb{C}^{r_1}$ coordinate and then factor out that constant by linearity: for any $\tau_L$ with $|\tau_L| = 1$, we have
$$
Q^L \left( \tau_L z_0^{* (1)} , \tau_L z_1^{* (1)} , \cdots , \tau_L z_{r_1-1}^{* (1)} , \;\;\;  z_0^{* (2)}, \cdots , z_{r_2-1}^{* (2)} , \cdots , z_{r_M-1}^{* (M)} \right) = \tau_L Q^L \left( ( z_i^{*(j)} ) \right) \; .
$$
Let $y_{L}$ be the point appearing in the equation above:
\begin{align*}
y_{i,j,L} &= \tau_L z_i^{*(j)} (L)  \qquad \text{ for } j = 1 \\
y_{i,j,L} &= z_i^{*(j)} (L)  \qquad \text{ for } j > 1
\end{align*}
for all $L$, and so $Q^L( y_L ) = \tau_L Q^L \left( ( z_i^{*(j)} ) \right) $.  We now fix $\tau_L$ such that
\begin{equation}\label{E:tau_L}
\beta_L Q^L( y_L ) = \beta_L \tau_L Q^L \left( ( z_i^{*(j)} ) \right)  = \left| Q^L \left( ( z_i^{*(j)} ) \right) \right| \; .
\end{equation}
We will now need Kronecker's Theorem \cite{Hardy_Wright}:
\begin{theorem}[Kronecker's Theorem]\label{T:Kronecker}
Let $\lambda_1, \ldots , \lambda_n \in \mathbb{R}$ be rationally independant.  Then, the map
$$
t \longrightarrow ( e^{i \lambda_1 t} , \ldots , e^{i \lambda_n t} )
$$
from $\mathbb{R}$ to $\mathbb{T}^n$ has dense range in $\mathbb{T}^n$.
\end{theorem}
The set $\{ \log p : p \; \text{is prime } \}$ is rationally independant, so by Kronecker's Theorem let $t_K$ be chosen so that
$$
| p_{k_i^{(j)}}^{-i t_K} - y_{i,j,L} | \le 2^{- (2M+1) L}
$$
is satisfied for all $i,j$ and all $L \le L_K$ (simultaneously).  Write this as
$$
p_{k_i^{(j)}}^{-i t_K} = y_{i,j,L} + \delta_{i,j,L} \; .
$$
So,
\begin{align*}
&P_L(i t_K) = Q^L( p_1^{-i t_K} , p_2^{-i t_K} , \ldots )  \\
&=  \sum_{i_1 , \ldots , i_M} \left( \tau_L z_{i_1}^{* (1)} + \delta_{i_1,1,L} \right)  \left( z_{i_2}^{* (2)} + \delta_{i_2,2,L} \right)  \cdots \left( z_{i_M}^{* (M)} + \delta_{i_M,M,L} \right)  \omega_{r_2}^{ i_1 i_2} \cdots  \omega_{r_M}^{ i_{M-1} i_M }  \\
&=  \sum_{i_1 , \ldots , i_M} \left( \tau_L z_{i_1}^{* (1)} z_{i_2}^{* (2)}  \cdots z_{i_M}^{* (M)}  + \sum_{\delta} \right) \omega_{r_2}^{ i_1 i_2} \cdots  \omega_{r_M}^{ i_{M-1} i_M }  
\end{align*}
where $\sum_{\delta}$ represents the remaining terms in the binomial expansion; we have
$$
\left| \sum_{\delta} \right| \le 2^M 2^{ -(2M+1)  L } \; .
$$
Continuing the above equality,
\begin{align*}
P_L(i t_K ) &= \tau_L Q^L \left( ( z_i^{*(j)} ) \right) + \sum_{i_1 , \ldots , i_M} \left( \sum_{\delta} \right) \omega_{r_2}^{ i_1 i_2} \cdots  \omega_{r_M}^{ i_{M-1} i_M }    \\
&= \tau_L Q^L \left( ( z_i^{*(j)} ) \right) + \delta_L  \; .
\end{align*}
The last line above serves to define $\delta_L$, and we have 
$$
| \delta_L | \le 2^{ (\rho_1 + \cdots + \rho_M) L} 2^M 2^{- (2M+1) L } \le 2^{M L} 2^M 2^{- (2M+1) L } \le 2^{-L} \; .
$$

Let $N_K = \max \{ n : n \in \Pi_{L_K}^{\times} \}$.  Then
\begin{align*}
\sum_{n=1}^{N_K} a_n n^{-i t_K} &= \sum_{L=1}^{L_K} \beta_L 2^{-X L} P_L(i t_K) \\
&= \sum_{L=1}^{L_K} \beta_L 2^{-X L} \left(  \tau_L Q^L \left( ( z_i^{*(j)} ) \right) + \delta_L \right) \\
&=  \sum_{L=1}^{L_K} \beta_L 2^{-X L} \tau_L Q^L \left( ( z_i^{*(j)} ) \right) + \sum_{L=1}^{L_K} \beta_L 2^{-X L} \delta_L  \\
&=  \sum_{L=1}^{L_K} 2^{-X L} \left| Q^L \left( ( z_i^{*(j)} ) \right) \right| + \sum_{L=1}^{L_K} \beta_L 2^{-X L} \delta_L  \qquad \text{[by (\ref{E:tau_L})]} \; .
\end{align*}
We see that
$$
\sum_{L=1}^{L_K} 2^{-X L} \left| Q^L \left( ( z_i^{*(j)} ) \right) \right|  \ge \sum_{L=1}^{L_K} 2^{- X L } 2^{-M} 2^{ ( \rho_1 + \cdots + \rho_M ) \frac{M+1}{2M} L} \; .
$$
We now can choose the value of $X$ that we would like:
$$
X = ( \rho_1 + \cdots + \rho_M ) \frac{M+1}{2M} \; .
$$
With this,
$$
\sum_{L=1}^{L_K} 2^{-X L} \left| Q^L \left( ( z_i^{*(j)} ) \right) \right|  \ge 2^{-M} L_K \ge 2 K
$$
and $ \left| \sum_{L=1}^{L_K} \beta_L 2^{-X L} \delta_L \right| \le \sum 2^{-X L} 2^{-L} \le 1 \le K$, so we have
$$
\left| \sum_{n=1}^{N_K} a_n n^{-i t_K}  \right| \ge (2K) - (K) = K \; .
$$
For arbitrary $K$, we have found a partial sum of $f$ which, at a point on the imaginary axis, assumes a value having modulus larger than $K$ and so we have proved

\begin{proposition}\label{P:sigma_b_lowerbound}

Let $f$ be the Dirichlet series defined by  (\ref{E:MainSeries}) and (\ref{E:Define_a_n}). 
With 
$$
X = ( \rho_1 + \cdots + \rho_M ) \frac{M+1}{2M}
$$
and for any choice of $\beta_L$, $|\beta_L|=1$, we have $\sigma_b \ge 0$.

\end{proposition}

\section{A Key Estimate}\label{S:General_InnerSum}

Before presenting the final bound (the upper bound on $\sigma_c$) we will require a size estimate on a sum of the following form:
$$
\sum_{n \in \Pi_L^{\times} , n \le P }  \omega_{r_2}^{ i_1 i_2}  \cdots  \omega_{r_M}^{ i_{M-1} i_M } 
$$
for general P.  This sum is adding terms on the unit circle with widely varying arguments, so a high degree of cancellation can be hoped for.  Unfortunately, no sophisticated or impressive bound has been obtained by the current author; we will simply isolate the $i_M$ index, sum the resulting one-variable geometric series, and then bound by absolute values.  Recalling
$$
n = p_{k_{i_1}^{(1)}} \cdots p_{k_{i_M}^{(M)}} \; \longleftrightarrow \; (i_1, \ldots ,i_M)
$$
the one key observation is that, because we selected the primes in increasing order, the set
$$
\left\{ (i_1, \ldots ,i_M) : n \le P \right\}
$$
has the following weak convexity property: if we fix $(i_1, \ldots ,i_{M-1})$, then the set 
$$
\left\{ i_M :  (i_1, \ldots ,i_M) \text{ satisfies } n \le P \right\}
$$
is an ``interval'' of natural numbers, meaning that it equals every natural number between some (unspecified) lower and upper bounds, call them $l$ and $u$ respectively.  Leaving out the terms with $i_{M-1} = 0$, we have:
\begin{align*}
\sum_{n \in \Pi_L^{\times} , n \le P \; , \; i_{M-1} \ne 0 }  &\omega_{r_2}^{ i_1 i_2}  \cdots  \omega_{r_M}^{ i_{M-1} i_M }  \\
&= \sum_{(i_1, \ldots ,i_{M-1}) \; , \;  i_{M-1} \ne 0 } \omega_{r_2}^{ i_1 i_2}  \cdots  \omega_{r_{M-1}}^{ i_{M-2} i_{M-1} }   \; \sum_{i_M :  (i_1, \ldots ,i_M) \text{ satisfies } n \le P} \omega_{r_M}^{ i_{M-1} i_M } \\
&= \sum_{(i_1, \ldots ,i_{M-1}) \; , \;  i_{M-1} \ne 0} \omega_{r_2}^{ i_1 i_2}  \cdots  \omega_{r_{M-1}}^{ i_{M-2} i_{M-1} }   \; \sum_{i_M = l}^u \omega_{r_M}^{ i_{M-1} i_M } \\
&= \sum_{(i_1, \ldots ,i_{M-1}) \; , \;  i_{M-1} \ne 0} \omega_{r_2}^{ i_1 i_2}  \cdots  \omega_{r_{M-1}}^{ i_{M-2} i_{M-1} }   \; \left(  \frac{\omega_{r_M}^a - \omega_{r_M}^b}{1 - \omega_{r_M}^{i_{M-1}} }  \right) \; . \\
\end{align*}
Noting that on $[-\pi, \pi]$ there is some small $c$ such that $1 - \cos x \ge c x^2$, we have
\begin{align*}
| 1 - \omega_{r_M}^{i_{M-1}} |^2 &= 2 \big(1 - \cos( 2 \pi i_{M-1} /r_M ) \; \big) \\
&\ge 2 c (2 \pi i_{M-1} /r_M)^2
\end{align*}
and so there is an absolute constant $c'$ such that $| 1 - \omega_{r_M}^{i_{M-1}} | \ge c' i_{M-1} /r_M$.

Now, bounding with absolute values, including those terms with $i_{M-1} = 0$, and estimating $\sum_{i=1}^K i^{-1} \le C \log (K+1)$, we have
\begin{align*}
\left| \sum_{n \in \Pi_L^{\times} , n \le P }  \omega_{r_2}^{ i_1 i_2}  \cdots  \omega_{r_M}^{ i_{M-1} i_M }  \right| &\le c'^{-1} \sum_{(i_1, \ldots ,i_{M-2})} 2 \sum_{i_{M-1} \ge 1} \frac{r_M}{i_{M-1}} + \sum_{(i_1, \ldots ,i_M) : i_{M-1} = 0} 1 \\
&\le 2 c'^{-1} C \; r_1 \cdots r_{M-2} r_M \big(  \log (r_{M-1}+1) \;  + 1 \big) \; .
\end{align*}
Note that $r_j \le r_{M+1}^{\rho_j} \le 2^{\rho_j L }$.  With $C'_M$ being another constant depending on $M$,
\begin{equation}\label{E:EstimatedOmegaSum}
\left| \sum_{n \in \Pi_L^{\times} , n \le P }  \omega_{r_2}^{ i_1 i_2}  \cdots  \omega_{r_M}^{ i_{M-1} i_M }  \right|  \le C'_M 2^{ ( \rho_1 + \cdots + \rho_{M-2} + \rho_M ) L }  L \; .
\end{equation}

\section{Convergence on the negative real axis}\label{S:General_sigma_c}

Note that, after proving the bounds on the abscissae of $f$ in the preceding sections, we still have the freedom to choose $\beta_L$.  We will now use $\beta_L$ to arrange a large amount of cancellation in the partial sums of $f$ at a certain point $s = - \epsilon$ (to be determined) on the negative real axis.

Fix some $\epsilon > 0$, and consider a partial sum of the series (\ref{E:MainSeries}) at $s = -\epsilon$:
$$
A_N(\epsilon) = \sum_{n = 1}^N a_n n^{\epsilon} \;.
$$
Our goal is to find for which values of $\epsilon$ it is possible to choose $\beta_L$ such that the sequence $\{ A_N(\epsilon) \}$ converges as $N \rightarrow \infty$.

A straightforward but quite crucial fact will be the following (mentioned previously): If $L_1 < L_2$ then
$$
\forall n_1 \in \Pi_{L_1}^{\times} \; , \; \forall n_2 \in \Pi_{L_2}^{\times}  \; , \; \text{we have} \;\; n_1 < n_2 \; .
$$
We define
$$
L^*(N) = \max \{ L : \exists n \le N \text{ with }  n \in \Pi_L^{\times}  \} \; .
$$
By definition of $L^*(N)$, and by the fact above, if $L<L^{*}(N)$ then
$$
\forall  n \in \Pi_L^{\times}  \; , \; \text{we have } n \le N
$$
and therefore we can write
$$
A_N(\epsilon) = \sum_{L<L^*(N)} \; \sum_{ n \in \Pi_L^{\times}  } a_n n^{\epsilon} + \sum_{ n \in \Pi_{L^*(N)}^{\times}  , n \le N} a_n n^{\epsilon} \; .
$$
The important point is that, in the first term, the inner sums range over all $n \in \Pi_L^{\times}$ because all these $n$ satisfy $n \le N$.

We would like to continue to express the inner sums in $A_N(\epsilon)$ as one-dimensional (in order to sum by parts in the desired order), so let proceed with the understanding 
$$
n = p_{k_{i_1}^{(1)}} \cdots p_{k_{i_M}^{(M)}} \;\; \longleftrightarrow \;\; (i_1, \ldots ,i_M) \; .
$$

We have
\begin{align*}
A_N(\epsilon) = \sum_{L < L^*(N)} &\beta_L 2^{-X L} \sum_{n \in \Pi_L^{\times} } \omega_{r_2}^{ i_1 i_2}  \cdots  \omega_{r_M}^{ i_{M-1} i_M }  \;  n^{\epsilon}  \\
&+  \beta_{L^*(N)} \alpha_{L^*(N)} \sum_{ n \in \Pi_{L^*(N)}^{\times}  , n \le N}  \omega_{r_2}^{ i_1 i_2}  \cdots  \omega_{r_M}^{ i_{M-1} i_M }   n^{\epsilon}   \; .
\end{align*}
Define
\begin{align}
\Omega_L(\epsilon) &= \sum_{n \in \Pi_L^{\times} } \omega_{r_2}^{ i_1 i_2}  \cdots  \omega_{r_M}^{ i_{M-1} i_M }   \; n^{\epsilon} \; .  \label{E:BigOmega} \\
\Gamma(N, \epsilon) &= \sum_{ n \in \Pi_{L^*(N)}^{\times}  , n \le N}  \omega_{r_2}^{ i_1 i_2}  \cdots  \omega_{r_M}^{ i_{M-1} i_M }  \; n^{\epsilon}  \; . \label{E:BigGamma}
\end{align}
This means
\begin{equation} \label{E:Main_A_N}
A_N(\epsilon) = \sum_{L < L^*(N)} \beta_L 2^{-X L} \Omega_L(\epsilon)  \;\;\; + \;\;\; \beta_{L^*(N)} \alpha_{L^*(N)} \Gamma(N, \epsilon) \; .
\end{equation}
We aim to show that $2^{-X L} | \Omega_L(\epsilon) | \rightarrow 0$ and $2^{-X L} | \Gamma(N, (\epsilon) | \rightarrow 0$, for a certain value of $\epsilon$.  This is not sufficient by itself, but along with our freedom to choose $\beta_L$ it will be enough to prove the result.

We prove next that the sums (\ref{E:BigOmega}), (\ref{E:BigGamma}) have a large amount of cancellation.  A large amount of cancellation should not be surprising, because of the $\omega^{i_1 i_2 + \cdots}$ factors.  Let us index the elements of $\Pi_L^{\times}$ in increasng order: 
$$
\Pi_L^{\times} = \left\{ n_1 , \ldots , n_{\left| \Pi_L^{\times}  \right|} \right\} \; .
$$  
We will sum by parts:

\begin{proposition}[A Version of Summation by Parts]

Suppose $a_1 , \ldots , a_p$ and $b_1 , \ldots , b_p$ are given.  Define $B_N = \sum_{j=1}^N b_j$.  Then
$$
\sum_{j=1}^p a_i b_i =  \sum_{j=1}^{p-1} (a_j - a_{j+1}) B_j + a_p B_p
$$

\end{proposition}

\begin{proof}
Standard.
\end{proof}

Applying this to $\Omega_L(\epsilon)$, $\Gamma(N, \epsilon)$, we have
$$
\Omega_L(\epsilon) = \sum_{j=1}^{\left| \Pi_L^{\times}  \right| - 1} (n_j^{\epsilon} - n_{j+1}^{\epsilon} )  \left[ \sum_{n \in \Pi_L^{\times} , n \le n_j }  \omega_{r_2}^{ i_1 i_2}  \cdots  \omega_{r_M}^{ i_{M-1} i_M }  \right]  \; + n_{\left| \Pi_L^{\times}  \right|}^{\epsilon}  \sum_{n \in \Pi_L^{\times} }  \omega_{r_2}^{ i_1 i_2}  \cdots  \omega_{r_M}^{ i_{M-1} i_M } 
$$
and, letting $n^*$ be the maximal element of $\left\{ n :  n \in \Pi_{L^*(N)}^{\times}  , n \le N  \right\}$,
\begin{align*}
\Gamma(N, \epsilon) =& \sum_{j : n_j \le N} (n_j^{\epsilon} - n_{j+1}^{\epsilon} )  \left[ \sum_{n \in \Pi_{L^*(N)}^{\times} , n \le n_j }  \omega_{r_2}^{ i_1 i_2}  \cdots  \omega_{r_M}^{ i_{M-1} i_M }  \right]  \\
& \;\;\; + (n^*)^{\epsilon}  \sum_{n \in \Pi_{L^*(N)}^{\times}  , n \le N}  \omega_{r_2}^{ i_1 i_2}  \cdots  \omega_{r_M}^{ i_{M-1} i_M }  \; .
\end{align*}
Now, we need to estimate the size of a sum of the form
$$
\sum_{n \in \Pi_L^{\times} , n \le P }  \omega_{r_2}^{ i_1 i_2}  \cdots  \omega_{r_M}^{ i_{M-1} i_M } 
$$
for general $P$, this estimate was provided by equation (\ref{E:EstimatedOmegaSum}) in the preceding section (we will substitute the actual bound in place of $E_L$ when required):
$$
\left| \sum_{n \in \Pi_L^{\times} , n \le P }  \omega_{r_2}^{ i_1 i_2}  \cdots  \omega_{r_M}^{ i_{M-1} i_M }  \right| \le E_L \; .
$$
With this estimate, we have
$$
\left| \Omega_L(\epsilon)  \right| \le 2 \; (\max \{ n : n \in \Pi_L^{\times} \})^{\epsilon} E_L \; .
$$
Recalling the bound (\ref{E:bound_n_L}) on the largest element of $\Pi_L^{\times}$, we see that, with $C_M$ being (another) constant depending on $M$, and $\epsilon<1$ assumed, we have
\begin{equation}\label{E:EstimateOmegaL}
\left| \Omega_L(\epsilon)  \right| \le C_M  2^{\epsilon M L} L^M \; E_L
\end{equation}
and similarly
\begin{equation}\label{E:EstimateGammaN}
\left| \Gamma(N, \epsilon)  \right| \le C_M  2^{\epsilon M L^*(N)} L^*(N)^M \; E_{L^*(N)} \; .
\end{equation}
Using the estimate (\ref{E:EstimatedOmegaSum}) on $E_L$, we substitute into (\ref{E:EstimateOmegaL}), (\ref{E:EstimateGammaN}), and we have
\begin{align*}
\left| \Omega_L(\epsilon)  \right| &\le C_M  2^{\epsilon M L} L^M \; 2^{ ( \rho_1 + \cdots + \rho_{M-2} + \rho_M ) L }  L \\
\left| \Gamma(N, \epsilon)  \right| &\le C_M  2^{\epsilon M L^*(N)} L^*(N)^M \; 2^{ ( \rho_1 + \cdots + \rho_{M-2} + \rho_M ) L^*(N)}  L^*(N) \; .
\end{align*}
Recall that (\ref{E:Main_A_N}) tells us
$$
A_N(\epsilon) = \sum_{L < L^*(N)} \beta_L 2^{-X L} \Omega_L(\epsilon) \;\;\; + \;\;\; \beta_{L^*(N)} \alpha_{L^*(N)} \Gamma(N, \epsilon) \; .
$$
We define $\beta_L = \beta_L(\epsilon)$ to satisfy
$$
\beta_L \Omega_L(\epsilon) = (-1)^{d_L} \left| \Omega_L(\epsilon)  \right| 
$$
where $d_L= d_L(\epsilon)$ is not yet specified.  This means we have
$$
A_N(\epsilon) = \sum_{L < L^*(N)}   (-1)^{d_L} 2^{-X L} \left| \Omega_L(\epsilon)  \right| \;\;\;  + \;\;\; \beta_{L^*(N)} \alpha_{L^*(N)} \Gamma(N, \epsilon) \; .
$$
Leaving aside the $L^*(N)$ term for the moment, recall the following result from infinite series: If $a_n \ge 0$ and $a_n \rightarrow 0$, then there exists some choice of signs $d_n$ such that the partial sums $\sum_{n \le N} (-1)^{d_n} a_n$ converge (to some unspecified value) as $N \rightarrow \infty$.  This means that, if (for some particular value of $\epsilon$)
$$
2^{-X L} \left| \Omega_L(\epsilon)  \right|  \rightarrow 0 \;\; \text{as } L \rightarrow \infty
$$
then there exists a choice of signs $d_L$ such that $\sum_{L < L^*(N)}   (-1)^{d_L} 2^{-X L} \left| \Omega_L(\epsilon)  \right| $ converges as $N \rightarrow \infty$.  Using the inequality above, we have
\begin{align*}
2^{-X L} \left| \Omega_L(\epsilon)  \right| &\le 2^{-X L} L^{-(M+2)} C_M  2^{\epsilon M L} L^M \; 2^{ ( \rho_1 + \cdots + \rho_{M-2} + \rho_M ) L } L \\
&= C_M 2^{ [ \; ( \rho_1 + \cdots + \rho_{M-2} + \rho_M) - X + \epsilon M \;]L}  L^{-1} \; .
\end{align*}
We see that, if $\epsilon$ satisfies
\begin{equation}\label{E:RhoEpsilonConditionFromSigmaC}
( \rho_1 + \cdots + \rho_{M-2} + \rho_M) - X + \epsilon M = 0
\end{equation}
then $2^{-X L} \left| \Omega_L(\epsilon)  \right|  \rightarrow 0$.  Additionally, since $\alpha_{L^*(N)} \Gamma(N,\epsilon)$ is bounded by the same quantity as above (with $L^*(N)$ substituted for $L$), if (\ref{E:RhoEpsilonConditionFromSigmaC}) is satisfied then we also have
$$
\alpha_{L^*(N)} \Gamma(N,\epsilon) \rightarrow 0
$$
and therefore there will be a choice of signs $\{d_L\}$ (i.e. a choice of $\{ \beta_L \}$) such that $A_N(\epsilon)$ converges as $N \rightarrow \infty$.

Let us choose
$$
\epsilon = \frac{-1}{M} \big[ \; ( \rho_1 + \cdots + \rho_{M-2} + \rho_M) - X \; \big]
$$
This means (\ref{E:RhoEpsilonConditionFromSigmaC}) is satisfied, so as long as this $\epsilon$ is greater than zero,  we can choose $\{ \beta_L \}$ such that $f(s)$ converges at $s = - \epsilon$ on the negative real axis, and so we have proved

\begin{proposition}\label{P:sigma_c}
Let $f$ be the Dirichlet series defined by  (\ref{E:MainSeries}) and (\ref{E:Define_a_n}).  $f$ satisfies
$$
\sigma_c \le \frac{1}{M} \big( ( \rho_1 + \cdots + \rho_{M-2} + \rho_M) - X \big)
$$
if the quantity on the right hand side is negative (it would not be of interest otherwise).
\end{proposition}

\section{Results}\label{S:results}

Examining propositions \ref{P:sigma_b_upperbound}, \ref{P:sigma_a}, \ref{P:sigma_b_lowerbound} and \ref{P:sigma_c}, we see that with $X = ( \rho_1 + \cdots + \rho_M ) \frac{M+1}{2M}$, the requirement
$$
X \le (1/2) ( \rho_1 + \cdots + \rho_M + 1 )
$$
from proposition \ref{P:sigma_b_upperbound} is satisfied, and so we have a Dirichlet series $f$ which satisfies
\begin{align*}
\sigma_c &\le \frac{1}{2 M^2} \big[ (M-1) ( \rho_1 + \cdots + \rho_{M-2} + \rho_M) - (M+1) \rho_{M-1} \big] \\
\sigma_b &\ge 0 \\
\sigma_b &\le \frac{1}{2M} \left( 1 - \frac{ \rho_1 + \cdots + \rho_M }{M}  \right)  \\
\sigma_a &\ge \frac{M-1}{2 M^2} (\rho_1 + \cdots + \rho_M) 
\end{align*}
as long as the bound on $\sigma_c$ is less than zero, i.e.
$$
(M-1) ( \rho_1 + \cdots + \rho_{M-2} + \rho_M) - (M+1) \rho_{M-1} < 0 \; .
$$

Proving the results stated in the introduction, for $M=2$ and $M=3$, is now a matter of arithmetic:

With $M=2$, and $\rho_1, \rho_2 = 1$, we have $\sigma_b = 0$, $\sigma_a \ge 1/4$, and $\sigma_c \le -1/4$.

With $M=3$, we can set $\rho_2 = \rho_3 = 1$ and then
\begin{align*}
\sigma_a &\ge \frac{1}{9} ( \rho_1 + 2 ) \\
\sigma_b &\in [0 , \frac{1}{18} (1 - \rho_1 ) ] \\
\sigma_c &\le \frac{1}{9} (\rho_1 - 1) \; .
\end{align*}

\end{document}